\documentclass[a4paper,12pt,twoside,reqno]{amsart}

\usepackage[a4paper,margin=1.15in]{geometry}
\usepackage{amsmath,amssymb,amsthm,mathtools}
\usepackage{booktabs,tabularx}
\numberwithin{equation}{section}
\usepackage[utf8]{inputenc}
\usepackage{mathrsfs}
\usepackage{eucal}
\usepackage{microtype}
\usepackage{comment}
\usepackage{enumitem}

\makeatletter
\def\cr@pref#1:#2\@nil{#1}
\@namedef{cr@map@lem}{Lemma}        \@namedef{cr@map@lem@p}{Lemmas}
\@namedef{cr@map@thm}{Theorem}      \@namedef{cr@map@thm@p}{Theorems}
\@namedef{cr@map@thmAlph}{Theorem}  \@namedef{cr@map@thmAlph@p}{Theorems}
\@namedef{cr@map@prop}{Proposition} \@namedef{cr@map@prop@p}{Propositions}
\@namedef{cr@map@cor}{Corollary}    \@namedef{cr@map@cor@p}{Corollaries}
\@namedef{cr@map@rem}{Remark}       \@namedef{cr@map@rem@p}{Remarks}
\@namedef{cr@map@defn}{Definition}  \@namedef{cr@map@defn@p}{Definitions}
\@namedef{cr@map@setup}{Setup}      \@namedef{cr@map@setup@p}{Setups}
\@namedef{cr@map@sec}{Section}      \@namedef{cr@map@sec@p}{Sections}
\@namedef{cr@map@tab}{Table}        \@namedef{cr@map@tab@p}{Tables}
\@namedef{cr@map@eq}{Equation}      \@namedef{cr@map@eq@p}{Equations}
\@namedef{cr@map@eqn}{Equation}     \@namedef{cr@map@eqn@p}{Equations}
\@namedef{cr@eqref@eq}{} \@namedef{cr@eqref@eqn}{}
\newcommand{\cr@tsing}[1]{\@nameuse{cr@map@\cr@pref#1:\@nil}}
\newcommand{\cr@tref}[1]{
  \@ifundefined{cr@eqref@\cr@pref#1:\@nil}{\ref{#1}}{\eqref{#1}}}
\newcommand{\Cref}[1]{\cr@tsing{#1}~\cr@tref{#1}}
\makeatother

\theoremstyle{plain}
\newtheorem{theorem}{Theorem}[section]

\newtheorem{thmAlph}{Theorem}

\newtheorem{lem}[theorem]{Lemma}
\newtheorem{prop}[theorem]{Proposition}
\newtheorem{cor}[theorem]{Corollary}

\theoremstyle{definition}
\newtheorem{remark}[theorem]{Remark}

\newtheorem{rem}[theorem]{Remark}

\usepackage[dvipsnames]{xcolor}
\usepackage[pagebackref]{hyperref}
\hypersetup{
    colorlinks=true,
    linkcolor=NavyBlue,
    citecolor=TealBlue,
    filecolor=NavyBlue,
    urlcolor=magenta
}
\usepackage{bookmark}
\newcommand{\doi}[1]{\href{https://doi.org/#1}{\nolinkurl{doi:#1}}}

\newcommand{\CC}{\mathbf{C}}
\newcommand{\ZZ}{\mathbf{Z}}
\newcommand{\PP}{\mathbf{P}}
\newcommand{\OO}{\mathscr{O}}

\newcommand{\cL}{\mathcal{L}}
\newcommand{\cQ}{\mathcal{Q}}
\newcommand{\cU}{\mathcal{U}}

\DeclareMathOperator{\Sym}{Sym}
\DeclareMathOperator{\Gr}{Gr}
\DeclareMathOperator{\rk}{rk}

\title[The Second-Type Locus]{The Universal Cover of the Second-Type Locus of a Cubic Fourfold}

\author{Frank Gounelas}
\address{Mathematisches Institut, Universit\"at Bonn, Endenicher Allee 60, 53115 Bonn, Germany}
\email{gounelas@math.uni-bonn.de}

\date{}

\begin{document}

\begin{abstract}
We prove that the surface of second-type lines on a general cubic fourfold has fundamental group of order two. Its universal cover, constructed by Huybrechts, is obtained by considering the two ramification points of the Gauss map along each second-type line.
\end{abstract}

\maketitle

\section{Introduction}\label{sec:introduction}

Let \(V\) be a six-dimensional complex vector space, let
\(X\subset\PP(V)\) be a smooth cubic fourfold, and let
\(F(X)\subset \Gr(2,V)\) denote its
Fano variety of lines, a hyperk\"ahler fourfold by \cite{BD85}. Following the classical
analysis of lines on cubic hypersurfaces in \cite[Section~6]{ClemensGriffiths},
a line \(L\subset X\) is called one of the following, depending on the splitting type of its normal bundle
\begin{align*}
N_{L/X}&\simeq\OO_{\PP^1}(1)\oplus\OO_{\PP^1}^{\oplus 2}
&&\text{(first-type)},\\[4pt]
N_{L/X}&\simeq\OO_{\PP^1}(1)^{\oplus 2}\oplus\OO_{\PP^1}(-1)
&&\text{(second-type)}.
\end{align*}
The locus of second-type lines can be highly singular even when \(X\)
is smooth, e.g., for the Fermat cubic fourfold it has \(55\) irreducible
components \cite[Theorem~A]{GK-Fermat}.
Assume now that \(X\) is general, and write \(S=F_2(X)\subset F(X)\)
for the locus of second-type lines. It appears as
the indeterminacy locus of Voisin's rational self-map \cite[Section~4.2]{Voisin-KCorrespondences}
\[
F(X)\dashrightarrow F(X),
\]
and was initially
studied by Amerik \cite{Amerik} and by Amerik--Voisin \cite{AmerikVoisin}.
The second-type locus is smooth and projective by \cite[Lemma~1]{Amerik} and irreducible by \cite[Proposition~6.4.8]{Huybrechts-Book}.
Its numerical invariants were computed in \cite[Section~4]{GK-Invariants} (see also \cite[Section~6.4]{Huybrechts-Book}).
If \(H_S\) is the restriction of the Pl\"ucker
polarisation and \(K_S\) is its canonical line bundle, then
\begin{gather*}
H_S^2=315,\qquad K_S^2=2835,\qquad
c_2(T_S)=2565,\qquad \chi(S,\OO_S)=450,\\[4pt]
h^{1,0}(S)=0,\qquad h^{2,0}(S)=449,\qquad
h^{1,1}(S)=1665.
\end{gather*}

In this paper, we complete the picture of the basic invariants of \(F_2(X)\) by proving the following.

\begin{thmAlph}\label{thm:main}
Let \(X\subset\PP(V)\) be a general cubic fourfold. Then
\[
\pi_1\bigl(S\bigr)\simeq\ZZ/2.
\]
Its universal double cover \(\widetilde S\) parametrises a second-type line together with one
of the two ramification points of the Gauss map restricted to that line.
\end{thmAlph}

Some motivation for the above theorem in the broader context of the geometry of cubic fourfolds is an upcoming paper of Hartlieb--Huybrechts, where a geometric Torelli theorem is proven for general cubic fourfolds using the second-type locus, and the above is one step needed in their construction.

Let \(X\) be given by the vanishing of some equation \(g\), and let \(G\) be its associated symmetric
trilinear form, normalised by \(G(x,x,x)=g(x)\).
The key construction to prove the theorem is the isomorphism
\[
\widetilde S\simeq Z=
\bigl\{([x],[y])\in X\times X\mid G(x,y,-)=0\bigr\},
\]
where the ordered pair \(([x],[y])\) consists of the two ramification points
on the second-type line that they span. Writing \(G(x,y,-)\) as an element of \(V^\vee\) with respect to a dual basis \(e_i^\vee\), the six components
\(G(x,y,e_i)\) define divisors of bidegree \((1,1)\), and \(Z\) is their
smooth codimension-six complete intersection. Equivalently, it is cut out by
six hyperplane sections in the Segre embedding. The Lefschetz theorem then
identifies its fundamental group with that of \(X\times X\), which is trivial. Along the way, we give two more (heavily related)
determinantal descriptions of the double cover in \Cref{sec:incidence}, which are of independent interest. The first is
as the zero locus of a section of a vector bundle on the two-pointed incidence variety, and the second is as a degeneracy
locus in the universal family of lines. Neither of these two extra descriptions allows for a direct application of
Lefschetz-type theorems as in the $X\times X$ case.

\subsection*{Notation}

We work over \(\CC\). Throughout, \(V\) is a six-dimensional vector space,
\(X\subset\PP(V)\) is a smooth cubic fourfold, \(F=F(X)\) is its Fano
variety of lines, and \( S=F_2(X)\subset F \)
is the surface parametrising lines of second type.
For a vector bundle \(E\), we use the convention that \(\PP(E)\)
parametrises one-dimensional subspaces of its fibres, with tautological
line subbundle \(\OO_{\PP(E)}(-1)\subset\pi^*E\), where
\(\pi\colon\PP(E)\to B\) is the projection to the base of \(E\).

\subsection*{Acknowledgements}
I would like to thank Moritz Hartlieb and Daniel Huybrechts for helpful conversations and for asking about the fundamental group of the second-type locus in the first place. I am also indebted to Alexis Kouvidakis for explaining to me most things I know about cubics. This work was supported by the ERC Synergy Grant HyperK (ID 854361).

\subsection*{AI Disclosure} 
The OpenAI LLMs GPT-5.6-Sol and GPT-6.1-Sol were used in the writing and editing of this paper. More specifically, the embedding of \(\widetilde S\) into \(X\times X\) had already been attempted by the author as well as degeneration to the Fermat cubic, but was completed by an LLM. Most of the mathematical content of the paper was obtained after numerous prompts in VSCode with one of the mentioned models, but the exposition and choices are the author's own work.

\subsection*{Data availability}
No external datasets were generated or analysed in this study. All material supporting the results is contained within the article.

\section{The second-type surface and its double cover}\label{sec:preliminaries}

Write \( X=\{g=0\}\subset\PP(V), \dim V=6, \)
where \(g\in\Sym^3(V^\vee)\) is a cubic equation. Let \(G\) be the associated
symmetric trilinear form, normalised by
\(
G(x,x,x)=g(x).
\)
Explicitly,
\[
G(x,y,z)
=\frac{1}{6}\bigl(
g(x+y+z)-g(x+y)-g(x+z)-g(y+z)
+g(x)+g(y)+g(z)
\bigr).
\]
The form \(G\) also induces the contraction map
\[
G'\colon V\otimes V\longrightarrow V^\vee,
\qquad
G'(x\otimes y)=G(x,y,-).
\]
Under the inclusion \( F=F(X)\subset \Gr(2,V), \)
let \(\cU_F\) and \(\cQ_F\) be the restrictions to \(F\) of the universal
subbundle and quotient bundle on \(\Gr(2,V)\). For \(\ell\in F\), let
\(U_\ell\subset V\) be the two-dimensional subspace corresponding to
\(\ell\), and put \(Q_\ell=V/U_\ell\). The corresponding line in \(X\)
is \(L=\PP(U_\ell)\).

We now summarise Amerik's construction of \(S\). For each \(\ell\in F\),
define a linear map
\begin{align}
\beta_\ell\colon\Sym^2(U_\ell)&\longrightarrow Q_\ell^\vee\notag\\
uv&\longmapsto\bigl(\overline w\longmapsto G(u,v,w)\bigr)
\label{eqn:beta-definition}
\end{align}
for \(u,v\in U_\ell\), where \(w\in V\) is any lift of
\(\overline w\in Q_\ell\). This does not depend
on the choice of lift. Indeed, another lift is \(w+z\) for some
\(z\in U_\ell\), and
\[
G(u,v,w+z)-G(u,v,w)=G(u,v,z)=0,
\]
because \(G|_{U_\ell^{\otimes 3}}=0\), since \(g\) vanishes identically on
\(U_\ell\). These maps lead to a bundle map
\[
\beta\colon\Sym^2(\cU_F)\longrightarrow\cQ_F^\vee.
\]

Recall that the Gauss map of \(X\) is defined by
\begin{align*}
\gamma_X\colon X&\longrightarrow\PP(V^\vee)\\
[u]&\longmapsto[dg_u].
\end{align*}
More explicitly, choose linear coordinates \(x_1,\ldots,x_6\) on \(V\) and a
nonzero representative \(u=(u_1,\ldots,u_6)\) of \([u]\). Then
\[
dg_u=\sum_{i=1}^6\frac{\partial g}{\partial x_i}(u)\,dx_i\in V^\vee,
\qquad
dg_u(z)=\sum_{i=1}^6\frac{\partial g}{\partial x_i}(u)z_i
\]
for \(z=(z_1,\ldots,z_6)\in V\). Replacing \(u\) by \(\lambda u\)
multiplies \(dg_u\) by \(\lambda^2\), so \([dg_u]\) is independent of the
chosen representative. The projective hyperplane
\(\PP(\ker(dg_u))\subset\PP(V)\) cut out by \(dg_u\) is the embedded tangent
hyperplane to \(X\) at \([u]\).
If \([u]\in L\), then \(dg_u\) annihilates \(U_\ell\), because \( dg_u(z)=3G(u,u,z)=0\) for every \(z\in U_\ell.\)
Thus \(dg_u\) descends to a linear functional on \(Q_\ell\), that is,
an element of its dual. Moreover, for every
\(\overline w\in Q_\ell\) and any lift \(w\in V\), the definition of
\(\beta_\ell\) gives \( dg_u(w)=3G(u,u,w)=3\beta_\ell(u^2)(\overline w), \)
and the induced functional is \(3\beta_\ell(u^2)\). In particular, the restricted Gauss map is
\begin{align*}
\gamma_X|_L\colon\PP(U_\ell)&\longrightarrow\PP(Q_\ell^\vee),\\
[u]&\longmapsto[\beta_\ell(u^2)],
\end{align*}
and the values of \(\beta_\ell\) on the squares \(u^2\) therefore encode the
Gauss map along \(L\).

\begin{lem}\label{lem:second-type}
Let \(X\) be a smooth cubic fourfold, let \(\ell\in F(X)\), and
let \(L=\PP(U_\ell)\subset X\) be the corresponding line. Every line is
either of first or second type, with
\begin{align*}
N_{L/X}&\simeq
\OO_{\PP^1}(1)\oplus\OO_{\PP^1}^{\oplus 2}\\
N_{L/X}&\simeq
\OO_{\PP^1}(1)^{\oplus 2}\oplus\OO_{\PP^1}(-1).
\end{align*}
In the first case, \(\beta_\ell\) has rank three and is injective, and the
projective Gauss map \(\gamma_X\)
restricts to an isomorphism from \(L\) onto a smooth conic. In the
second case, \(\beta_\ell\) has rank two, and \(\gamma_X\) restricts to a finite morphism \( \gamma_X|_L\colon L\longrightarrow\gamma_X(L) \)
of degree two onto a projective line. Moreover, there exist linearly
independent \(a,b\in U_\ell\) such that \( \ker(\beta_\ell)=\CC(ab), \)
and \([a],[b]\) are the two
ramification points of \(\gamma_X|_L\).
\end{lem}
\begin{proof}
The normal-bundle splitting types and the two possibilities for the
restricted Gauss map are proved in \cite[Section~6]{ClemensGriffiths}.
Since the squares span \(\Sym^2(U_\ell)\) and the restricted Gauss map is
\([u]\mapsto[\beta_\ell(u^2)]\), the rank of \(\beta_\ell\) is three
in the first case and two in the second.
If \(\beta_\ell(a^2)=0\), then the functional
\(dg_a\), which already annihilates \(U_\ell\), induces the zero functional
on \(Q_\ell\). Hence \(dg_a=0\) on \(V\), contradicting the smoothness of
\(X\). Thus
\begin{equation}\label{eqn:no-square}
\beta_\ell(a^2)\ne0\qquad\text{for every }0\ne a\in U_\ell.
\end{equation}
Consequently, if
\(\beta_\ell\) has rank two, its one-dimensional kernel is spanned by a
nonsquare binary quadratic form, which over \(\CC\) factors as \(ab\) for linearly
independent \(a,b\in U_\ell\). Writing \(u=sa+tb\), we then have \( \beta_\ell(u^2) =s^2\beta_\ell(a^2)+t^2\beta_\ell(b^2). \)
In particular, in suitable coordinates, the restricted Gauss map is
\([s:t]\mapsto[s^2:t^2]\), whose ramification points are
\([a]\) and \([b]\).
\end{proof}

For a morphism \(\psi\) of vector bundles, we write \(D_r(\psi)\) for the
closed degeneracy locus where the fibre map has rank at most \(r\), defined
scheme-theoretically by the \((r+1)\times(r+1)\) minors of \(\psi\).

Amerik's construction also realises the second-type surface
scheme-theoretically as the degeneracy locus
\[
S=D_2(\beta)
=\{\ell\in F\mid\rk(\beta_\ell)\leq 2\}.
\]
We will use the following geometric properties.

\begin{lem}[{\cite[Lemma~1]{Amerik}, \cite[Proposition~6.4.8]{Huybrechts-Book}}]\label{lem:surface}
If \(X\) is a general cubic fourfold, the second-type locus \(S\)
is a smooth, projective, irreducible surface.
\end{lem}

As mentioned in the introduction, the second-type locus \(S\) has a natural double cover, which is constructed by considering the two ramification points of the Gauss map along each second-type line. We record this construction in the following lemma.

\begin{lem}[{\cite[Remark~6.4.9, p.~320]{Huybrechts-Book}}]\label{lem:gauss-cover}
Let \(X\) be a general cubic fourfold and define
\[
\eta=K_S\otimes\OO_S(-3H_S).
\]
Then \(\eta\) is a nontrivial 2-torsion line bundle, and the corresponding
connected \'etale double cover, i.e., with
\(p_*\OO_{\widetilde S}\simeq\OO_S\oplus\eta\), is the surface \(p\colon\widetilde S\longrightarrow S\) 
\[
\widetilde S=
\left\{(\ell,r)\ \middle|\;
r\text{ is a ramification point of }
\gamma_X|_{\PP(U_\ell)}
\right\} \subset S\times X.
\]
\end{lem}

\section{Three descriptions of the Gauss cover}\label{sec:incidence}

We begin with the description of \(\widetilde S\) that will be used to
compute its fundamental group. Recall the contraction map
\(G'\colon V\otimes V\to V^\vee\), and set
\[
M=X\times X,
\qquad
\cL=\OO_X(1)\boxtimes\OO_X(1),
\]
and consider the bundle morphism induced by the contraction map
\begin{align}
\OO_M(-1,-1)&\longrightarrow V^\vee\otimes\OO_M,\notag\\
x\otimes y&\longmapsto G'(x\otimes y)=G(x,y,-).
\label{eqn:contraction-section}
\end{align}
After tensoring by \(\cL\), this morphism gives a section
\(s_{G'}\in H^0(M,\cL\otimes V^\vee)\).
Its zero scheme is defined by the condition
\(G(x,y,-)=0\).

\begin{prop}\label{prop:contraction-zero-locus}
Let \(X\) be a general cubic fourfold. The zero scheme
\(Z=Z(s_{G'})\subset X\times X\), defined by \(G(x,y,-)=0\), is
isomorphic to \(\widetilde S\). In particular, it is a smooth complete
intersection of six divisors of bidegree \((1,1)\).
\end{prop}
\begin{proof}
We first describe the zero scheme under the Segre embedding.
Choose a basis \(e_1,\ldots,e_6\) of \(V\), with dual basis
\(e_1^\vee,\ldots,e_6^\vee\). Then
\[
s_{G'}=\sum_{i=1}^6s_i\otimes e_i^\vee,
\qquad
s_i\in H^0(M,\cL),
\]
where \(s_i\) is induced by the bilinear form
\((x,y)\mapsto G(x,y,e_i)\).
Thus \(Z(s_{G'})\) is the zero scheme of the six sections
\(s_1,\ldots,s_6\). The line bundle \(\cL\) is the restriction of
\(\OO_{\PP(V)\times\PP(V)}(1,1)\), whose complete linear system defines the
Segre embedding \( \PP(V)\times\PP(V)\longrightarrow\PP(V\otimes V)\simeq\PP^{35}, \)
and the components of \(s_{G'}\) are restrictions of linear forms
on this ambient projective space.

The six bilinear forms \(G(-,-,e_i)\) are linearly independent.
Indeed, a nontrivial linear relation among them would give a nonzero vector
\(w\in V\) with \(G(w,-,-)=0\). In particular, \(g(w)=G(w,w,w)=0\)
and \(dg_w=3G(w,w,-)=0\), contradicting the smoothness of \(X\).
Their corresponding linear forms on \(\PP(V\otimes V)\) define a
projective linear subspace \(\Lambda\subset\PP^{35}\) of codimension six
whose scheme-theoretic intersection \(M\cap\Lambda\) is cut out in \(M\) by
the restrictions of the linear forms vanishing on \(\Lambda\), i.e., \(s_1,\ldots,s_6\). As these sections also cut out the zero
scheme \(Z(s_{G'})\), the two ideal sheaves agree, and
\[
Z(s_{G'})=M\cap\Lambda
\]
as closed subschemes of \(M\).

We next prove that this intersection is a smooth surface. Consider the incidence correspondence
\begin{multline*}
\mathcal{Z}=\bigl\{(g,[x],[y])\ \big|\ g(x)=g(y)=0,\ G(x,y,-)=0\bigr\}\\
\subset\PP\bigl(\Sym^3(V^\vee)\bigr)\times\bigl((\PP(V)\times\PP(V))\setminus\Delta\bigr).
\end{multline*}
Writing \(x_i=e_i^\vee\),
at the point \(([e_1],[e_2])\) the eight defining equations are the vanishing of the coefficients of
the eight monomials
\[
x_1^3,\quad x_2^3,\quad x_1^2x_2,\quad x_1x_2^2,\quad
x_1x_2x_3,\quad\ldots,\quad x_1x_2x_6.
\]
Indeed, \(g(e_1)\) and \(g(e_2)\) are the coefficients of \(x_1^3\) and
\(x_2^3\), while \(G(e_1,e_2,e_k)\) is a nonzero multiple of the
coefficient of \(x_1x_2x_k\) for \(1\le k\le6\). Since \(\mathrm{PGL}(V)\) acts transitively on pairs of distinct
points, these eight linear conditions are independent at every point of
\((\PP(V)\times\PP(V))\setminus\Delta\). Thus
\(\mathcal{Z}\) is a \(\PP^{47}\)-bundle over the smooth irreducible variety
\((\PP(V)\times\PP(V))\setminus\Delta\), and is itself smooth and irreducible
of dimension \(57\).
By Lemmas~\ref{lem:second-type} and~\ref{lem:surface}, the projection \(\mathcal{Z}\to\PP(\Sym^3(V^\vee))\) is dominant, so
generic smoothness gives smooth general fibres of dimension \(2\).
For a general cubic \(g\), this fibre is
\((M\cap\Lambda)\setminus\Delta\), and \(M\cap\Lambda\) is disjoint from the
diagonal: indeed, \(G(x,x,-)=0\) would make \([x]\) a singular point of
\(X\). Thus \(M\cap\Lambda\) is a smooth surface.

Finally, we construct the isomorphism with \(\widetilde S\). Since \(Z(s_{G'})\) is reduced, it
suffices to produce mutually inverse morphisms between the two varieties.
Now suppose that \([x]\ne[y]\) and \(s_{G'}([x],[y])=0\). We shall show that
the projective line spanned by \([x]\) and \([y]\) lies in \(X\), is of
second type, and has \([x]\) and \([y]\) as the two ramification points of
its restricted Gauss map. Since \([x],[y]\in X\) and \( G'(x\otimes y)=0, \) we have \( G(x,x,y)=G(x,y,y)=0.  \)
For all \(s,t\in\CC\), trilinearity gives
\begin{align*}
G(sx+ty,sx+ty,sx+ty)
={}&s^3G(x,x,x)+3s^2tG(x,x,y)\\
&+3st^2G(x,y,y)+t^3G(y,y,y)
=0.
\end{align*}
Thus the cubic vanishes identically on \( L=\PP(\CC x+\CC y).  \)
Let \(\ell\in F\) be the point corresponding to \(L\). The map
\(\beta_\ell\) kills \(xy\), so it is not injective. By
\Cref{lem:second-type}, \(L\) is of second type,
\(\ker(\beta_\ell)=\CC(xy)\), and \([x]\) and \([y]\) are the two
ramification points of \(\gamma_X|_L\).

Conversely, let \(\ell\in S\), write \(L=\PP(U_\ell)\), and let \([x]\) and
\([y]\) be the ramification points of \(\gamma_X|_L\). By
\Cref{lem:second-type}, \( \beta_\ell(xy)=0.  \) This equality says that the map \( G'(x\otimes y)\colon V\longrightarrow\CC, \)
which annihilates \(U_\ell=\CC x+\CC y\), induces the zero map on
\(V/U_\ell\). Hence \(G'(x\otimes y)=0\), and therefore \( s_{G'}([x],[y])=0.  \)

Let \(\iota\) be the deck involution of \(\widetilde S\to S\). The morphism
\[
\theta\colon\widetilde S\longrightarrow Z(s_{G'}),
\qquad
(\ell,[x])\longmapsto([x],[y]),
\quad
\iota(\ell,[x])=(\ell,[y]),
\]
is well defined by the preceding calculation. Conversely,
\(Z(s_{G'})\) is disjoint from the diagonal, and the span map
\[
\bigl(\PP(V)\times\PP(V)\bigr)\setminus\Delta
\longrightarrow\Gr(2,V),
\qquad
([x],[y])\longmapsto\CC x+\CC y,
\]
is a morphism. Its restriction to \(Z(s_{G'})\), together
with the first marked point, defines a morphism
\[
Z(s_{G'})\longrightarrow\widetilde S,
\qquad
([x],[y])\longmapsto(\ell,[x]).
\]
The two compositions agree with the identity on closed points, and
from reducedness, they agree with the identity as morphisms by
\cite[Exercise~5.17, p.~146]{GortzWedhorn}.
Thus the two morphisms are inverse.
\end{proof}

\begin{proof}[Proof of \Cref{thm:main}]
By \Cref{lem:gauss-cover},
\(p\colon\widetilde S\to S\) is a connected \'etale double cover whose
fibres are the two ramification points. Choose \(s\in S\) and
\(\widetilde s\in\widetilde S\) with \(p(\widetilde s)=s\). The covering
gives a short exact sequence
\begin{align}\label{eqn:cover-sequence}
1\longrightarrow\pi_1(\widetilde S,\widetilde s)
 \xlongrightarrow{\,p_*\,}\pi_1(S,s)
 \longrightarrow\ZZ/2
 \longrightarrow1.
\end{align}
It therefore suffices to prove that \(\widetilde S\) is simply connected.
Recall that \(M=X\times X\), and that
\Cref{prop:contraction-zero-locus} identifies \(\widetilde S\) with \( Z(s_{G'})=M\cap\Lambda, \)
where \(\Lambda\subset\PP^{35}\) is a projective linear subspace of
codimension six and \(M\subset\PP^{35}\) via the Segre embedding.

Since \(Z=M\cap\Lambda\) is a smooth complete intersection, we can choose
hyperplanes \(H_1,\ldots,H_6\) defining \(\Lambda\) such that every
successive intersection \(M_j=M\cap H_1\cap\cdots\cap H_j\) is smooth
of dimension \(8-j\). Indeed, the differentials of the six defining
equations are independent along \(Z\), and Bertini's theorem gives
smoothness away from \(Z\) for general choices of the hyperplanes.
The Lefschetz hyperplane theorem for homotopy groups
\cite[Theorem~I and its corollary]{Bott-Lefschetz} then gives an isomorphism
\(\pi_1(M_j)\simeq\pi_1(M_{j-1})\) for \(1\le j\le6\), where
\(M_0=M\), since \(\dim M_{j-1}\ge3\). Consequently, \( \pi_1(M\cap\Lambda)\xrightarrow{\sim}\pi_1(M).  \)
Since \(X\) is simply connected by the Lefschetz hyperplane theorem,
\(\pi_1(M)=1\).
By \Cref{prop:contraction-zero-locus}, \( M\cap\Lambda=Z(s_{G'})\simeq\widetilde S, \)
and we conclude that \( \pi_1(\widetilde S)=1.  \)
Consequently, \(\widetilde S\) is the universal cover and
\(\pi_1(S)\simeq\ZZ/2\).
\end{proof}

We now record the basic invariants of \(\widetilde S\),
completing the analogy with the list for \(S\) given in the introduction.
The first four are formal consequences of the covering, whereas the Hodge numbers
require \Cref{thm:main}. These can be computed either from the double cover \(p:\widetilde S\to S\) or from the zero-locus construction \(Z(s_{G'})\subset X\times X\) of
\Cref{prop:contraction-zero-locus}.

\begin{cor}\label{cor:invariants}
Let \(X\) be a general cubic fourfold, and endow \(\widetilde S\) with the
polarisation \(H_{\widetilde S}=p^*H_S\). Then
\begin{align*}
H_{\widetilde S}^2&=630,\qquad
K_{\widetilde S}^2=5670,\qquad
c_2(T_{\widetilde S})=5130,\qquad
\chi(\widetilde S,\OO_{\widetilde S})=900,\\[4pt]
h^{1,0}(\widetilde S)&=0,\qquad
h^{2,0}(\widetilde S)=899,\qquad
h^{1,1}(\widetilde S)=3330.
\end{align*}
Moreover, \(H^1(S,\eta)=0\) and \(H^2(S,\eta)\) is \(450\)-dimensional.
\end{cor}

We record now two more bundle constructions for \(\widetilde S\). Consider first the fibre product
\[
J=\PP(\cU_F)\times_F\PP(\cU_F),
\]
parametrising \(\ell\in F\) together with two ordered marked points
\([x],[y]\in L=\PP(U_\ell)\).

We remark that the variety \(J\) is irreducible and simply connected. Indeed, the
Beauville--Donagi theorem gives that \(F(X)\) is an irreducible hyperk\"ahler
manifold of \(K3^{[2]}\)-type \cite{BD85}, and the morphism \(J\to F\) is a
locally trivial \(\PP^1\times\PP^1\)-bundle. The homotopy exact sequence
therefore gives \(\pi_1(J)=\pi_1(F)=1\).

Denote the two
tautological line subbundles on \(J\) by \(\OO_J(-1,0)\) and
\(\OO_J(0,-1)\), and continue to denote by \(\cQ_F^\vee\) its pullback to
\(J\).
The map \(\beta\) defines a bundle morphism
\begin{align*}
\sigma\colon\OO_J(-1,-1)&\longrightarrow\cQ_F^\vee,\\
\sigma_{(\ell,[x],[y])}(x\otimes y)&=\beta_\ell(xy).
\end{align*}

Let \(\rho\colon J\to X\times X\) be the evaluation morphism
\((\ell,[x],[y])\mapsto([x],[y])\). After tensoring the subbundle
inclusion \(\cQ_F^\vee\hookrightarrow V^\vee\otimes\OO_J\) by
\(\OO_J(1,1)=\rho^*\cL\), the definitions
\eqref{eqn:beta-definition} and~\eqref{eqn:contraction-section} show
that the section corresponding to \(\sigma\) maps to \(\rho^*s_{G'}\) since
both associated bundle maps send \(x\otimes y\) to \(G(x,y,-)\).
In particular \(Z(\sigma)\) is the
scheme-theoretic inverse image of \(Z(s_{G'})\) under \(\rho\).
The following lemma shows that \(\rho\) restricts to an isomorphism
between these zero schemes.

\begin{lem}\label{lem:two-pointed}
Let \(X\) be a general cubic fourfold. The evaluation morphism \(\rho\)
restricts to an isomorphism of schemes
\(Z(\sigma)\xrightarrow{\sim}Z(s_{G'})\simeq\widetilde S\).
\end{lem}
\begin{proof}
Put \(M^\circ=(X\times X)\setminus\Delta\), and consider the span morphism
\[
\operatorname{span}\colon M^\circ\longrightarrow\Gr(2,V),
\qquad
([x],[y])\longmapsto\CC x+\CC y.
\]
Set \(A=\operatorname{span}^{-1}(F(X))\subset M^\circ\), with the
scheme-theoretic inverse-image structure.
If \(J^\circ\subset J\) denotes the open locus where the two marked points
are distinct, then the restriction of \(\rho\) induces an isomorphism
\[
\rho|_{J^\circ}\colon J^\circ\xrightarrow{\sim}A,
\qquad
(\ell,[x],[y])\longmapsto([x],[y])
\]
with inverse
\(([x],[y])\mapsto(\operatorname{span}([x],[y]),[x],[y])\).
By
\eqref{eqn:no-square}, \(Z(\sigma)\) is disjoint from the diagonal and is
therefore contained in \(J^\circ\).

It remains only to prove that \(Z(s_{G'})\) is contained
scheme-theoretically in \(A\), so that then the two vanishing loci agree on the nose.
As shown in the proof of \Cref{prop:contraction-zero-locus}, \(Z(s_{G'})\)
is disjoint from the diagonal.
Work locally on \(M^\circ\), choosing representatives
\(x=\sum_kx_ke_k\) and \(y=\sum_ky_ke_k\). In the coordinate ring of
the corresponding chart of \(Z(s_{G'})\), the six defining equations
\(G(x,y,e_k)=0\) give
\[
G(x,x,y)=\sum_kx_kG(x,y,e_k)=0,
\qquad
G(x,y,y)=\sum_ky_kG(x,y,e_k)=0.
\]
Together
with \(g(x)=g(y)=0\), they are the four coefficients of the restriction
\[
g(sx+ty)=s^3g(x)+3s^2tG(x,x,y)+3st^2G(x,y,y)+t^3g(y).
\]
Thus all four coefficients of \(g(sx+ty)\) vanish in
\(\OO_{Z(s_{G'})}\). The vanishing of these coefficients is precisely
the condition defining \(F(X)\) as a closed subscheme of \(\Gr(2,V)\).
Consequently, the restriction of the span morphism to \(Z(s_{G'})\)
defines a morphism to \(F(X)\). Equivalently, the ideal defining its
scheme-theoretic inverse image \(A\) vanishes in \(\OO_{Z(s_{G'})}\).
Hence \(Z(s_{G'})\) is a closed subscheme of \(A\).
It follows that
\[
Z(\sigma)
=J^\circ\times_{M^\circ}Z(s_{G'})
\simeq A\times_{M^\circ}Z(s_{G'})
=Z(s_{G'}).
\]
The rest now follows from \Cref{prop:contraction-zero-locus}.
\end{proof}

We finally record perhaps the most natural description of the degree-two
cover \({\widetilde S}\).
Let
\[
\pi\colon I=\PP(\cU_F)\longrightarrow F
\]
be the universal family of lines, and write
\(\OO_I(-1)\subset\pi^*\cU_F\)
for the tautological line subbundle. Define \( \phi\colon\OO_I(-1)\otimes\pi^*\cU_F\longrightarrow\pi^*\cQ_F^\vee \)
to be the composite
\[
\begin{array}{ccccc}
\OO_I(-1)\otimes\pi^*\cU_F
&\xrightarrow{\ \mu\ }&
\pi^*\Sym^2(\cU_F)
&\xrightarrow{\ \pi^*\beta\ }&
\pi^*\cQ_F^\vee,\\
u\otimes v
&\xmapsto{\hphantom{\ \mu\ }}&
uv
&\xmapsto{\hphantom{\ \pi^*\beta\ }}&
\beta_\ell(uv)
\end{array}
\]
where \((\ell,[u])\in I\), with \(\ell\in F\), \(u\in U_\ell\setminus\{0\}\), and \(v\in U_\ell\).

\begin{lem}\label{lem:incidence-degeneracy}
Let \(X\) be a general cubic fourfold. Then \( D_1(\phi)\simeq\widetilde S.  \)
\end{lem}
\begin{proof}
Fix \(\ell\in F\) and \([u]\in\PP(U_\ell)\). The fibre map
\(\phi_{(\ell,[u])}\) is the restriction of \(\beta_\ell\) to the
two-dimensional subspace \( uU_\ell\subset\Sym^2(U_\ell).  \)
By \eqref{eqn:no-square}, \(\beta_\ell(u^2)\ne0\). Hence \(\phi\) never vanishes and
\(D_0(\phi)=\varnothing\).

Let \(q\colon J\to I\) be the projection forgetting the second marked point.
Under the natural identification \( J\simeq \PP_I\bigl(\OO_I(-1)\otimes\pi^*\cU_F\bigr), \)
the point \((\ell,[u],[v])\) corresponds to the line generated by
\(u\otimes v\), and \(\OO_J(-1,-1)\) is the tautological line
subbundle of \(q^*\bigl(\OO_I(-1)\otimes\pi^*\cU_F\bigr)\). Consequently, \(Z(\sigma)\) is the zero scheme of the
 evaluation of \(\phi\), namely
\[
Z(\sigma)
=
\left\{(\ell,[u],[v])\in J\ \middle|\;
u\otimes v\in\ker\phi_{(\ell,[u])}\right\}.
\]
Equivalently, its fibre over \((\ell,[u])\in I\) is \( Z(\sigma)\cap q^{-1}(\ell,[u]) = \PP\bigl(\ker\phi_{(\ell,[u])}\bigr).  \)

Since \(\phi_{(\ell,[u])}\) has a two-dimensional source, this fibre is
nonempty precisely when \( \rk\phi_{(\ell,[u])}\leq1, \)
that is, precisely when \((\ell,[u])\in D_1(\phi)\). Moreover,
\(D_0(\phi)=\varnothing\), so \(\phi\) has rank exactly one along
\(D_1(\phi)\). Its kernel is therefore one-dimensional, and the above
projective space consists of a single point. It remains to check that this
identification is scheme-theoretic.

This may be checked locally on \(I\). Since \(D_0(\phi)=\varnothing\), the
open subsets on which some entry of a matrix for \(\phi\) is invertible cover
\(I\). After row and column operations, the matrix on such an open subset has
the form
\[
\begin{pmatrix}
1 & 0\\
0 & b_2\\
0 & b_3\\
0 & b_4
\end{pmatrix}
\qquad\text{and hence}\qquad
D_1(\phi)=V(b_2,b_3,b_4).
\]
If \([r:s]\) are the homogeneous coordinates associated with the two columns
of this matrix, the zero scheme \(Z(\sigma)\) is defined by
\[
r=0,\qquad b_2s=b_3s=b_4s=0.
\]
The equation \(r=0\) forces \(s\ne0\), so on the chart \(s=1\), $Z(\sigma)=V(r,b_2,b_3,b_4)$.
Hence \(q\) restricts scheme-theoretically to an isomorphism \( Z(\sigma)\xrightarrow{\sim}D_1(\phi).  \)
Combining this with \Cref{lem:two-pointed} proves the result.
\end{proof}

\begin{rem}\label{rem:not-ample}
It is natural to try to compute \(\pi_1(\widetilde S)\) directly either from
the zero-locus construction \(Z(\sigma)\subset J\) of
\Cref{lem:two-pointed} or from the degeneracy-locus construction
\(D_1(\phi)\subset I\) of \Cref{lem:incidence-degeneracy}. Both \(I\) and
\(J\) are simply connected, and the Lefschetz-type theorems for degeneracy
loci and for zero loci of sections of ample vector bundles (see \cite{FL81})
have exactly the right numerical range. For instance, if
\(\OO_J(1,1)\otimes\cQ_F^\vee\) were ample, they would give
\[
\pi_i\bigl(J,Z(\sigma)\bigr)=0 \text{ for }
i\le\dim J-\rk\bigl(\OO_J(1,1)\otimes\cQ_F^\vee\bigr)=2.
\]
The required ampleness, however, fails for both constructions.

Indeed, fix \(p=[x]\in
X\); the family of lines through \(p\) is positive-dimensional, and we let
\(C_p\subset F\) be a curve in this family. Along
the section \(\ell\mapsto(\ell,[x],[x])\) of \(J\to F\) over \(C_p\) both
tautological subbundles are trivial, so the restriction of
\(\OO_J(1,1)\otimes\cQ_F^\vee\) to this section is isomorphic to
\(\cQ_F^\vee|_{C_p}\). Its determinant is the restriction of the inverse
Pl\"ucker line bundle \(\OO_F(-1)\), which has negative degree on \(C_p\).
Thus \(\OO_J(1,1)\otimes\cQ_F^\vee\) is not ample. The same curves obstruct
the ampleness of
\(\OO_I(1)\otimes\pi^*(\cU_F^\vee\otimes\cQ_F^\vee)\) from
\Cref{lem:incidence-degeneracy}. \Cref{prop:contraction-zero-locus} avoids
these obstructions, since the corresponding bundle on \(X\times X\) is the
ample \(\bigl(\OO_X(1)\boxtimes\OO_X(1)\bigr)^{\oplus6}\).
\end{rem}

\end{document}